\documentclass[a4paper]{article}

\author{P. Del Moral \footnote{Centre INRIA Bordeaux Sud-Ouest
\& Institut de Math\'ematiques de Bordeaux
Universit\'e Bordeaux I
351, cours de la Lib\'eration
33405 Talence cedex, France}, F. Patras\footnote{CNRS UMR 6621, 
Universit\'e de Nice, Laboratoire de Math\'ematiques J.-A. Dieudonn\'e,  
Parc Valrose, 
06108 Nice Cedex 2, 
France}, S. Rubenthaler\footnote{CNRS UMR 6621, 
Universit\'e de Nice, Laboratoire de Math\'ematiques J.-A. Dieudonn\'e,  
Parc Valrose, 
06108 Nice Cedex 2, 
France - Tel. : (+33)04.93.16.87.90}}
\usepackage{amsmath}
\usepackage{amssymb}
\usepackage{amsthm}

\usepackage{bbm}
\usepackage{enumerate}
\usepackage{makeidx}

\usepackage[active]{srcltx}

\newtheorem{theorem}{Theorem}[section]
\newtheorem{lemma}[theorem]{Lemma}

\newtheorem{definition}[theorem]{Definition}

\newtheorem{corollary}[theorem]{Corollary}

\numberwithin{equation}{section}

\newcommand{\N}{\mathbb{N}}

\newcommand{\Q}{\mathbb{Q}}

\newcommand{\E}{\mathbb{E}}

\newcommand{\p}{\mathbb{P}}
\newcommand{\R}{\mathbb{R}}

\renewcommand{\epsilon}{ \varepsilon}

\newcommand{\Bsym}{\mathcal{B}^{\text{sym}}_0}

\title{Convergence of $U$-statistics for interacting particle systems.}

\begin{document}

 \maketitle     
 
\begin{abstract}
The convergence of $U$-statistics has been intensively studied for estimators based on families of i.i.d. random variables and variants of them. In most cases, the independence assumption is crucial \cite{lee-1990,delapena-gine-1999}. When dealing with Feynman-Kac and other interacting particle systems of Monte Carlo type, one faces a new type of problem. Namely, in a sample of $N$ particles obtained through the corresponding algorithms, the distributions of the particles are correlated -although any finite number of them is asymptotically independent with respect to the total number $N$ of particles. In the present article, exploiting the fine asymptotics of particle systems, we prove convergence theorems for $U$-statistics in this framework.

\end{abstract}

\section*{Introduction}

The convergence of $U$-statistics has been intensively studied for estimators based on families of i.i.d. random variables and variants of them. In most cases, the independence assumption is crucial. When dealing with Feynman-Kac and other interacting particle systems of Monte-Carlo type, one faces a new type of problem. Namely, in a sample of $N$ particles obtained through the corresponding algorithms, the distributions of the particles are correlated -although any finite number of them is asymptotically independent with respect to the total number $N$ of particles. It happens so (and this is the main contribution of the present article to show) that this asymptotic independence is enough in practice to insure the convergence of $U$-statistics based on interacting particle systems. In the following, we prove therefore the convergence of $U$-statistics for different particle systems under mild assumption that are satisfied by Feynmann-Kac particle systems. The case of Bird and Nanbu systems also fits in this framework and will be treated elsewhere by the third Author \cite{rubenthaler-2009}. To study the asymptotics of Feynman-Kac systems, whose properties are crucial to ensure the convergence of the statistics, we will use a functional representation, as introduced in \cite{DPR-2006} in the framework of discrete interacting particle systems.

The article is organized as follows. To fix the notations and the general framework of interacting particle systems, we first recall the (Feynman-Kac, continuous) interacting particle model. The model first appeared in quantum physics, in the work of Feynman and Kac in the 1940-50's, as a way to encode the motion of a quantum particle evolving in a potential (e.g. the interaction potential of a quantum field theory, viewed as a perturbation of the free Hamiltonian) in terms of path-integral formulas. It was realized progressively that interacting particle systems could be used in incredibly many different settings in probability and statistics.
A detailed list of the (still expanding) application areas of these models is contained in \cite{del-moral-2004}, to which we refer for further informations. 
Recall simply, since we focus here on its statistical features, that the model is mainly used, in applied statistics, as a Bayesian nonlinear filtering model: the motion of the particles is driven by a diffusion process and the potential encodes the likewood of the states with respect to observations or to some reference path. 

We study then the associated empirical joint distributions of a finite number $k$ of particles and study the convergence of the distributions in terms of the total number $N$ of particles of the system. This is closely related to our previous joint work \cite{DPR-2006} on discrete Feynman-Kac models -although the continuous hypothesis we use in the present article leads to some simplification of the tricky combinatorics that showed up in the discrete framework. 

We turn then to $U$-statistics for interacting particle systems and prove that under mild asumptions on the behavior of the system (satisfied e.g. by Feynman-Kac and Boltzmann systems) several asymptotic normality properties holds.

\section{Feynman-Kac particle systems}\label{Sec:1}
Let us consider a $E$-valued Markov process $X_t$, where $E=\R^d$ (or an arbitrary metric space) with a time-inhomogeneous infinitesimal generator $L_t$, continuous trajectories, and a positive bounded potential function $V_t$, $0\leq V_t(x)\leq V_\infty$. We assume that the distribution of $X_0$ is $\gamma_0= \eta_0$. Notice that these hypothesis are meaningful for most applications, but could be accomodated to more general ones, see \cite{del-moral-2004}. We are interested in the unnormalized (resp. normalized) distribution flows $\gamma_t$ and $\eta_t$ that are solutions, for sufficiently regular test functions $f$ and under appropriate regularity conditions, of the nonlinear equations:
$$\frac{d}{dt}\gamma_t(f)=\gamma_t(L_t(f))-\gamma_t(fV_t)$$
and
$$\frac{d}{dt}\eta_t(f)=\eta_t(L_t(f))+\eta_t(f(\eta_t(V_t)-V_t)).$$
In terms of $X_t$, we have:
$$
\gamma_t(f) = \E\left( f(X_t) \exp\left(-\int_0^t V_s(X_s) ds\right)\right),\ \eta_t(f)=\frac{\gamma_t(f)}{\gamma_t(1)}\ . 
$$
\subsection{Definitions and notations}

Let us fix first of all some notations.
For $q\in \N^*$, we write $[q]:=\{1,\dots,q\}$. For $q,N \in \N^*$, we set $\langle q,N \rangle := \{s\in [N]^{[q]}, s\ \text{injective}\}$.
For $q$ even, we write $\mathcal{I}_q$ for the set of  partitions of $[q]$ in pairs. We have $$\# \mathcal{I}_q = \frac{q!}{2^{q/2}\left( \frac{q}{2} \right)!}\ .$$
The set of smooth bounded  (resp. smooth bounded symmetric) functions on $E^q$ is written $
\mathcal{B}_b ( E^q)$ (resp. $
\mathcal{B}_b^{\text{sym}}( E^q)$). We also write $\mathcal{B}^{\text{sym}}_0 ( E^q)$ for the set of symmetric functions:
$$\mathcal{B}^{\text{sym}}_0 ( E^q)  := \left\{ F \in \mathcal{B}^{\text{sym}}_b ( E^q) : \int_{E} F(x_1, \dots, x_q) \gamma_t(dx_q) = 0\right\}.$$
Notice that the set of functions $\mathcal{B}^{\text{sym}}_0 ( E^q)$ depends on $t$, so that a better notation would be $\mathcal{B}^{\text{sym}}_{0,t} ( E^q)$. However, since in practice the abbreviated notation should not lead to confusion, we decided not emphasize this dependency for notational simplicity. We write simply $\mathcal{B}_0 ( E)$ for the centered functions: $F \in \mathcal{B}_b ( E) : \int_{E} F(x) \gamma_t(dx) = 0$

The empirical (possibly random) measure associated to a (possibly random) vector $x=(x_1,\dots,x_N )\in E^N$ is given by

$$
m(x)=\frac{1}{N} \sum_{i=1}^N \delta_{x_i}\ .
$$
We have, for all $q\in \N^*$, $F:E^q \rightarrow \R$:
$$
m(x)^{\otimes q}(F) = \frac{1}{N^q} \sum_{s\in [N]^{[q]}} F(x_{s(1)},\dots,x_{s(q)})\ .
$$
We also consider the corresponding $U$-statistics:
$$
m(x)^{\odot q}(F) = \frac{1}{(N)_q} \sum_{s \in \langle q , N\rangle } F(x_{s(1)},\dots,x_{s(q)})
$$
$$={{N}\choose{q}}^{-1}\sum\limits_{1\leq i_1<...<i_q\leq N}F_{sym}(x_{i_1},\dots,x_{i_q})\ ,
$$
where 
$
(N)_q=\frac{N!}{(N-q)!}=\# \langle q, N\rangle\ 
$ and
$
(F)_{\text{sym}}(y_1,\dots,y_q) := \frac{1}{q!} \sum_{\sigma \in S_q} F(y_{\sigma(1)},\dots,y_{\sigma(q)})\ ,
$
with $S_q$ the symmetric group of order $q$.
  
The notion of differential for sequences of signed measures\footnote{From now on, \it measure \rm will have the more general meaning of \it signed measure\rm} will also be useful.
The total variation norm is written $||~~||_{tv}$, so that, for any linear operator $L$ on $\mathcal{B}_{b}(E^q)$,
$$||L||_{tv}:=\sup_{f\in\mathcal{B}_b(E^q):\Vert f\Vert\leq 1}~|L(f)|$$
Let $(\Theta^N)_{N\geq 1}$,  be a uniformly bounded sequence of measures on $E^q$, in the sense that $\sup_{N\geq 1}{\|\Theta^N\|_{\rm{ } tv}}<\infty$. 
The sequence
$\Theta^N$ is said to converge strongly to some measure $\Theta$, 
as $N\uparrow\infty$ if and only if
 $$
\forall f\in \mathcal{B}_b(E)\qquad
\lim_{N\uparrow\infty} \Theta^N(f)= \Theta(f)
$$ 
\begin{definition} Let us assume that $\Theta^N$ converges strongly to $\Theta$.
The discrete derivative of the sequence $(\Theta^N)_{N\geq 1}$ is the sequence of
measures $ (\partial \Theta^N)_{N\geq 1}$ defined by
 $$
 \partial \Theta^N:=N~\left[ \Theta^N-\Theta\right]
 $$
We say that $\Theta^N$ is differentiable, if $ \partial \Theta^N$
is uniformly bounded, and if it 
 converges strongly to some measure $\partial \Theta$, called simply the derivative of $\Theta^N$,
as $N\uparrow\infty$.
\end{definition}

The discrete derivative $\partial\Theta^N$ of a differentiable 
sequence  can itself be differentiable.  In this situation, the derivative of the discrete derivative is called the second derivative of $\Theta^N$ and it is denoted by $ \partial^2 \Theta=\partial  \left(\partial\Theta \right)$, and so on. 

A sequence $\Theta^N$ that is
 differentiable up to order $(k+1)$, has the following  representation
$$
\Theta^N=\sum_{0\leq l\leq k }\frac{1}{N^l}~\partial^l\Theta+\frac{1}{N^{k+1}}~\partial ^{k+1}\Theta^N
$$
with $\sup_{N\geq 1}{\|\partial ^{k+1}\Theta^N\|_{\rm{ } tv}}<\infty$
, and the convention $\partial^0\Theta=\Theta$, for $l=0$.
  
\subsection{A genetic particle model}

A particle system approaching the measures $\eta_t$ (and therefore also $\gamma_t$) is the following. At $t=0$, the random vector $\Xi_t=(\xi^1_t,\dots,\xi^N_t)$ is a family of i.i.d. random variables distributed according to $\gamma_0=\eta_0$. Each entry $\xi^i_t$ of the vector diffuses according to the generator $L_t$, independently of the other entries.  Each entry $\xi_t^i$ has an exponential clock of parameter $V_\infty$ (independent of all the other variables defined up to now). When the clock of $\xi^i_t$ rings, say at $\tau$, it can 
\begin{itemize}
\item  jump to a randomly (and uniformly) chosen particle in the family, including itself, with a probability $\frac{V_\tau(\xi_t^i)}{V_\infty}$
\item stay where it is with probability $1- \frac{V_\tau(\xi_t^i)}{V_\infty}$.
\end{itemize}
Up to a renormalization of the jump type generators (which is convenient for our present purposes), this is the model described in \cite[Sect. 1.5.2]{del-moral-2004}.
Notice that the law of $\xi_t^i$ depends on $N$, so that a more consistent notation for $\xi_t^i$ would be $\xi_t^{N,i}$ -whenever we want to emphasize the dependency on $N$, we will switch to this second notation.
The corresponding empirical\footnote{The adjective \it empirical \rm refers, in the present article, to any measure, process or statistics obtained from a particle system approximation.} measures and empirical U-statistics are given (and related) by:
$$\eta_t^N =m(\Xi_t)=m(\xi^1_t,\dots, \xi^N_t),$$
$$\gamma_t^N =\gamma_t^N(1)\cdot\eta_t^N, \text{with} \ \gamma_t^N(1)=\exp\left(  -\int_0^t \eta_s^N(V_s) ds \right)\ ,$$ 
$$(\gamma_t^N)^{\odot q}(F)=\gamma^N_t(1)^q\cdot(\eta_t^N)^{\odot q}(F).$$
These are the analogs in continuous time to the random measures and U-statistics defined and studied in \cite{DPR-2006} in a discrete time setting. 
We recall that $\eta_t^N$ and $\gamma_t^N$ are known to converge as $N\to \infty$ to $\eta_t$ and $\gamma_t$ \cite{del-moral-2004}.

In oder to study the convergence of the (empirical) U-statistics $(\gamma_t^N)^{\odot q}$ and $(\eta_t^N)^{\odot q}$, we are going to rewrite them by means of functional expansions (actually, Laurent series in the parameter $N$). These expansions will allow us, later, to control the convergence and (among others) to extend to particle systems the classical central limits theorems for U-statistics.
Deriving these expansions is the main purpose of the first part of the article, together with first results of convergence.

For these purposes, it is useful to introduce an auxiliary $q$-particle system. The reasons for its introduction will become clear later and stem from a backward analysis of Feynman-Kac trajectories of families of $q$ particles.

\begin{definition}\label{Def:aux}
The \it auxiliary system \rm of  $q$ particles, $\hat{\xi}^1_t,\dots,\hat{\xi}^q_t$ is defined as follows:
the random variables $\hat{\xi}^1_0,\dots,\hat{\xi}^q_0$ are independant of law $\eta_0$. Moreover, the particles
\begin{equation}
 \label{Def:hatX1}
\hat{\xi}^1_t,\dots,\hat{\xi}^q_t  \text{ diffuse according to }  L_t
\end{equation}
(one more time, independently of one another)
and undergo the following jumps. For any $(i,j)\in [q]^2,\ i\not= j$, there is an exponential clock of parameter  $V_\infty/N$ and a corresponding Poisson point process $T^{(i,j)}_1,T^{(i,j)}_2,\dots$ of parameter $V_\infty/N$. The $T^{(i,j)}_1,T^{(i,j)}_2,\dots$ are named the ringing times.
 At a ringing  time $t\in \{ T^{(i,j)}_1,T^{(i,j)}_2,\dots\}$,
\begin{equation}
\label{Def:hatX2}
\hat{\xi}^i_t 
\begin{cases}
 \leftarrow \hat{\xi}^j_t & \text{ with  proba. } \frac{V_t(\hat{\xi}^i_t)}{V_\infty}\\
\leftarrow \hat{\xi}^j_t & \text{ with  proba. } 1-\frac{V_t(\hat{\xi}^i_t)}{V_\infty}\ .
\end{cases}
\end{equation}
\end{definition}

The notation $\hat{\xi}^i_t \leftarrow \hat{\xi}^j_t$ means that $\hat{\xi}^i_t$ jumps to (or is substituted by) $\hat{\xi}^j_t$. When $q=N$, the particle systems $(\hat{\xi}^1_t,\dots,\hat{\xi}^N_t)_{t\geq 0}$ has the same law as  $({\xi}^1_t,\dots,{\xi}^N_t)_{t\geq 0}$.

\

We set, for an arbitrary $F\in \mathcal{B}_b ( E^q)$,
$$
F^e((\hat{\xi}^1_s,\dots,\hat{\xi}^q_s)_{0\leq s \leq t}) =  F(\hat{\xi}^1_t,\dots,\hat{\xi}^q_t) \exp\left( {-\int_0^t [V_s(\hat{\xi}^1_s) +\dots + V_s(\hat{\xi}^q_s)] ds }\right)\ .
$$
and
$$
E_{t,k}(F) := \E\left\{ F^e((\hat{\xi}^1_s,\dots,\hat{\xi}^q_s)_{0\leq s \leq t})| k\ \text{rings on} [0;t] \right\}
\ .$$
Notice, for further use, that $E_{t,k}(F)$ does not depend on $N$. This is because, as a consequence of the general properties of Poisson point processes, conditionally to the hypothesis that there are $k$ rings on $[0,t]$, the distribution of the ringing times is uniform on $[0,t]$ and therefore independent of the parameter  $V_\infty/N$; since the other parameters of the process ($L_t$ and the jump probabilities $\frac{V_t(\hat{\xi}^i_t)}{V_\infty}$) are independent of $N$, the property follows.

\subsection{Expansion of the unnormalized measure}\label{unnorm}

Let us write for an arbitrary $F\in \mathcal{B}_b ( E^q)$, 
$$\mathbb{Q}^N_{t,q}(F):=\E ((\gamma_t^N)^{\odot q}(F)).$$ 
We refer to $\mathbb{Q}^N_{t,q}$ as to the empirical unnormalized measure associated to the particle system $\xi_t^i$. Since the joint distribution of the sequence $\xi_t^1,\dots,\xi_t^N$ is invariant by permutations, $\mathbb{Q}^N_{t,q}(F)=\mathbb{Q}^N_{t,q}(F_{sym})$, and we can assume without restriction that $F$ is a symmetric function.
The first question we adress is the (exact) computation of the speed of convergence of the empirical unnormalized measure $\mathbb{Q}^N_{t,q}$ to $\gamma_t^{\otimes q}$.

\begin{theorem}
\label{Theo:FKunnormalized}
For $F\in \mathcal{B}_b^{sym} ( E^q)$, we have the Laurent expansion
\begin{eqnarray}
\label{Eq:FKunnormalized1}
 \mathbb{Q}^N_{t,q}(F) &=& 
 \sum_{k=0}^\infty \frac{(\lambda t)^k e^{-\lambda t}}{k!} E_{t,k}(F)\\
\label{Eq:FKunnormalized2}
&=& \gamma_t^{\otimes q}(F) +\underset{k+i\geq 1}{\sum_{k,i\geq 0}^\infty}  (-1)^i\frac{(\lambda t)^{k+i}}{k!i!} E_{t,k}(F)\\
&=& \gamma_t^{\otimes q}(F) +\sum_{r=1}^\infty \frac{1}{N^r}\sum_{k=0}^r\left( \frac{(-1)^{r-k}}{k!(r-k)!} \left( q(q-1) V_\infty t  \right)^r E_{t,k}(F)\right)
\end{eqnarray}
where $\lambda:={q(q-1)}\frac{V_\infty}{N}$.  
In particular,  $\mathbb{Q}^N_{t,q}$ is differentiable up to any order with 
$$
\partial^r \mathbb{Q}_{t,q}(F) = \sum_{k=0}^r \frac{(-1)^{r-k}}{k!(r-k)!} \left( {q(q-1)} V_\infty t  \right)^r E_{t,k}(F)\ .
$$
\end{theorem}

Notice in particular that, although the $E_{t,k}(F)$s depend on the choice of the upper bound $V_\infty$ for the potential function $V$, the coefficients of the development do not (they are the derivatives of $\mathbb{Q}_{t,q}$, that do not depend on $V_\infty$).

\proof[Proof of Theorem \ref{Theo:FKunnormalized}]

\

We have first, for $f\in \mathcal{B}_b ( E)$
$$\frac{d}{dt} \E (\eta_t^N(f))=\E(\eta_t^N(L_t(f)))+\sum\limits_{i,j=1}^N\E\left(\frac{V_t(\xi_t^i)\left( f(\xi_t^j)-f(\xi_t^i) \right)}{N^2}\right)$$
and
$$\frac{d}{dt}\E( \gamma_t^N(f))=\E(\gamma_t^N(L_t(f)))-\E(\eta_t^N(V_t)\gamma_t^N(f))+\E\left(\gamma_t^N(1)\sum\limits_{i,j=1}^N\frac{V_t(\xi_t^i)\left( f(\xi_t^j)-f(\xi_t^i) \right)}{N^2}\right)$$

For $F\in \mathcal{B}_b^{sym} ( E^q)$, let us introduce the useful notation:
$F_i:E^{q-1}\longrightarrow \mathcal{B}_b ( E),$
$$F_i(x_1,...,x_{q-1})(y):=F(x_1,...,x_{i-1},y,x_i,...,x_{q-1})$$
and let us extend $L_t$ to functions in $\mathcal{B}_b^{sym}$:
$$L_t(F)(x_1,...,x_n):=\sum\limits_{i=1}^qL_t(F_i(x_1,...,x_{i-1},x_{i+1},...,x_n))(x_i).$$

We get:

\begin{eqnarray}
\label{Eq:deriv1}
\lefteqn{
\frac{d}{dt} \E((\gamma_t^N)^{\odot q} (F)) }
\\
\nonumber
& = & \E( (\gamma_t^N)^{\odot q} (L_t (F))) - \E(q \eta_t^N(V_t) (\gamma_t^N)^{\odot q }(F))\\
\nonumber
&& + \E\left\{ \gamma_t^N(1)^q\sum_{s\in <q,N>} \sum_{i=1}^q   \sum_{k=1}^N V_t(\xi^{s(i)})\right. \\
\nonumber
&&\left. \times\left( \frac{F_i(\xi^{s(1)}_t,\dots,\xi^{s(i-1)}_t,\xi^{s(i+1)}_t,\dots, \xi^{s(q)}_t)(\xi_t^k) - F(\xi^{s(1)}_t,\dots,\xi^{s(i)}_t,\dots, \xi^{s(q)}_t)}{N\cdot(N)_q} \right)   \right\}
\end{eqnarray}

The second term in the right hand side of  (\ref{Eq:deriv1}) reads:
\begin{eqnarray*}
&&\E\left( -q\frac{\gamma_t^N(1)^q} {N\cdot(N)_q} \sum_{s\in <q,N>}\sum_{k=1}^N F(\xi^{s(1)}_t , \dots , \xi^{s(q)}_t)  V_t(\xi^k_t) \right)=\\
&&\E\left( -q\frac{\gamma_t^N(1)^q} {N\cdot(N)_q} \sum_{s\in <q,N>}F(\xi^{s(1)}_t , \dots , \xi^{s(q)}_t)  (V_t(\xi^{s(1)}_t)+\dots + V_t(\xi^{s(q)}_t)) \right)+\\
&&\E\left( -q\frac{\gamma_t^N(1)^q} {N\cdot(N)_q} \sum_{s\in <q+1,N>} F(\xi^{s(1)}_t , \dots , \xi^{s(q)}_t)  V_t(\xi^{s(q+1)}_t) \right).
\end{eqnarray*}

Similarly, making use of the symmetry properties of $F$, the last term in the right hand side of  (\ref{Eq:deriv1}) reads:

\begin{eqnarray*}
\E\left( q\frac{\gamma_t^N(1)^q}{N\cdot (N)_q}  \sum_{s\in <q,N>}   \sum_{k=1}^N V_t(\xi^{s(q)}_t) F(\xi^{s(1)}_t,\dots,\xi^{s(q-1)}_t,\xi^k_t)\right)-\\
\E\left( \frac{\gamma_t^N(1)^q}{(N)_q}  \sum_{s\in <q,N>}   F(\xi^{s(1)}_t,\dots,\xi^{s(q-1)}_t,\xi_t^{s(q)})(V_t(\xi^{s(1)}_t)+\dots+V_t(\xi^{s(q)}_t))\right)\\
\end{eqnarray*}

The first term in this last sum decomposes then into (making use once again of the symmetry properties of $F$):
\begin{eqnarray*}
\E\left(\frac{\gamma_t^N(1)^q}{N\cdot (N)_q}  \sum_{s\in <q,N>} (V_t(\xi^{s(1)}_t)+\dots+V_t(\xi^{s(q)}_t)) F(\xi^{s(1)}_t,\dots,\xi^{s(q-1)}_t,\xi^{s(q)}_t)\right)+\\
\E\left( q\frac{\gamma_t^N(1)^q}{N\cdot (N)_q}  \sum_{s\in <q,N>}   \sum_{k=1}^{q-1} V_t(\xi^{s(q)}_t) F(\xi^{s(1)}_t,\dots,\xi^{s(q-1)}_t,\xi^{s(k)}_t)\right)+\\
\E\left(q\frac{\gamma_t^N(1)^q}{N\cdot (N)_q}  \sum_{s\in <q+1,N>} F(\xi^{s(1)}_t,\dots,\xi^{s(q-1)}_t,\xi^{s(q)}_t)V_t(\xi^{s(q+1)}_t) \right)
\end{eqnarray*}

Reorganizing the summands in these expansions, we get finally that the two last terms of (\ref{Eq:deriv1}) sum up to :

\begin{eqnarray*}
&&  - \E\left(  \frac{\gamma_t^N(1)^q }{(N)_q} \sum_{s\in <q,N>} F(\xi^{s(1)}_t,\dots,\xi^{s(q)}_t) (V_t(\xi^{s(1)}_t)+\dots+V_t(\xi^{s(q)}_t))  \right)+\\
&&\E \Big{(} \frac{\gamma_t^N(1)^q}{N\cdot (N)_q} \sum_{s\in <q,N>}  \sum_{i,r=1}^q V_t(\xi^{s(i)}_t) \left[F(\xi^{s(1)}_t,\dots,\xi^{s(i-1)}_t,\xi^{s(r)}_t,\xi^{s(i+1)}_t,\dots,\xi^{s(q)}_t)\right.\\
&& ~~~~~~~~~~~~~~~~~~~~~~~~~~~~~~~~~~~~~~~~~~~~~~~~~  \left. - F(\xi^{s(1)}_t,\dots,\xi^{s(q)}_t)\right]   \Big{)}\ .
\end{eqnarray*}
We set
\begin{eqnarray*}
 F_{V_t}(x_1,\dots,x_q) &:=& (V_t(x_1)+\dots+V_t(x_q)) F(x_1,\dots,x_q) \\
F_{V_t}^{(i,r)}(x_1,\dots, x_q) &=& V_t(x_i) (F(x_1,\dots,x_{i-1},x_r,x_{i+1},\dots,x_q)-F(x_1,\dots,x_q)).
\end{eqnarray*}
Then the equations above give
$$
\frac{d}{dt}\mathbb{Q}^N_{t,q}(F) =  \mathbb{Q}^N_{t,q} L_t (F) -  \mathbb{Q}^N_{t,q}(F_{V_t}) + \sum_{1\leq i,r\leq q} \frac{1}{N} \mathbb{Q}^N_{t,q}(F_{V_t}^{(i,r)})\ . 
$$
And so
\begin{eqnarray}\label{fundam}
\mathbb{Q}^N_{t,q}(F) = \E\left(F(\hat{\xi}^1_t,\dots,\hat{\xi}^q_t) \exp\left( -\int_0^t V_t(\hat{\xi}^1_s) +\dots + V_t(\hat{\xi}^q_s) ds\right)\right) \ .\end{eqnarray}
We then obtain Equation (\ref{Eq:FKunnormalized1}) of the theorem by noticing that 
$$
\p(k \text{ rings on} [0,t] ) = \frac{(\lambda t)^k e^{-\lambda t}}{k!}\ .
$$
We obtain Equation (\ref{Eq:FKunnormalized2}) by developping the term $e^{-\lambda t}$ and by noticing that
$$
E_{t,0}(F) = \gamma_t^{\otimes q}(F)\ .
$$
\endproof

\subsection{Wick theorem for interacting particle systems}

We say that two particles of the auxiliary system $\hat\xi^i$, $\hat\xi^j$, $i\not=j$ interact at $t$ if and only if $\hat\xi^i_t$ jumps to $\hat\xi^j$ or $\hat\xi^j_t$ jumps to $\hat\xi^i$ at $t$. 
We say that a trajectory of the auxiliary system of particles $\hat\xi^1_t,\dots,\hat\xi^q_t$ is exactly Wick-coupled on $[0,t]$ if and only each particle of the system has exactly one interaction with another particle between $0$ and $t$ (notice that the existence of such trajectories requires $q$ to be even). We write $\mathcal{W}_t$ for the set of Wick-coupled trajectories on $[0,t]$; the set $W_t$ embeds into the set of trajectories with exactly $q/2$ rings on $[0,t]$.

\begin{theorem}\label{Thm:Wick1}
If $F \in \mathcal{B}^{\text{sym}}_0 ( E^q)$ then, for $r< q/2$, $\partial^r \mathbb{Q}^N_{t,q}(F) =0$.
Moreover, for $q$ even,  
\begin{eqnarray*}
 \partial^{q/2} \mathbb{Q}_{t,q}(F) & =&  \frac{({q(q-1)} V_\infty t)^{q/2}}{(q/2)!}E_{t,q/2}(F)
\end{eqnarray*}
In particular,
\begin{equation}
\label{Eq:majWick}
\limsup_{N\rightarrow +\infty} N^{q/2}|\mathbb{Q}^N_{t,q}(F)| \leq \frac{({q(q-1)} V_\infty t)^{q/2}}{(q/2)!} \Vert F\Vert_\infty .
\end{equation}
\end{theorem}

\begin{proof}[Proof of Thm \ref{Thm:Wick1}]
Recall that $F$ is symmetric and belongs to $\mathcal{B}^{\text{sym}}_0 ( E^q)$ so that, for any $i\leq q$, $\int_E F(x_1,...,x_q)\gamma_t(dx_i)=0$. 

The proof follows then from Thm~\ref{Theo:FKunnormalized} together with the observation that, if there are strictly less than $q/2$ rings, then at least one of the particles does not interact with the others so that, if $k<\frac{q}{2}$, $E_{t,k}(F)=0$. 
\end{proof}

\begin{corollary}
With the assumptions of Thm~\ref{Theo:FKunnormalized} and $q$ even, we have:
\begin{eqnarray*}
 \partial^{q/2} \mathbb{Q}_{t,q}(F) & =&  \frac{q!}{(q/2)!} (V_\infty t)^{q/2} \E((F^e(\hat{\xi}^1_s, \dots, \hat{\xi}^q_s)_{0\leq s \leq t}) | \mathcal{W}_t).
\end{eqnarray*}

If, furthermore, $F=(f^1 \otimes \dots \otimes f^q)_{\text{sym}}$ (with $f^1,\dots,f^q$ centered with respect to $\gamma_t$) then we have the Wick-type expansion
\begin{eqnarray*}
\E(F^e((\hat{\xi}^1_s, \dots, \hat{\xi}^q_s)_{0\leq s \leq t})| \mathcal{W}_t) = \frac{2^{q/2} (q/2)!}{q!} \sum_{I_q\in\mathcal{I}_q} \prod_{\{i,j\} \in I_q} E'_{t,1}(f^i \otimes f^j)
\end{eqnarray*}
where the sum is over the set ${\mathcal{I}_q}$ of partitions $I_q$ of $[q]$ into pairs, $i,j$ being paired if $\{i,j\} \in I_q$ and where, for a function $G$ of two variables
$$
E'_{t,1}(G) := \E(G(\hat{\mu}^1_t,\hat{\mu}^2_t)\exp\left(-\int_0^t [V_s(\hat{\mu}^1_s) + V_s(\hat{\mu}^2_s)] ds  \right)|\text{one ring on } [0,t])
$$
with $\hat{\mu}^1_t,\hat{\mu}^2_t$ an auxiliary system of two particles (defined with the same rules as a general auxiliary system of $q$ particles).

\end{corollary}

\begin{proof}
Indeed, from the enumeration of possible interactions between particles and the general properties of Poisson processes, we have:
$$\mathbb{P}(\mathcal{W}_t|\text{$q/2$ rings on}\ [0,t])=\frac{q!}{(q(q-1))^{q/2}}.$$
The first part of the Corollary follows.

The last part follows from the observation that the distributions of two particles (or of blocks of particles) are independent conditionnally to the assumption that they do not interact (either directly or through interactions with other particles), so that quantities such as $ \E(F^e((\hat{\xi}^1_s, \dots, \hat{\xi}^q_s)_{0\leq s \leq t})| \mathcal{W}_t) $ may be computed by disjoint integrations over blocks (pairs in this particular case) of interacting particles. Since $Card(\mathcal{I}_q)=\frac{q!}{2^{q/2} (q/2)!}$, the Corollary follows.
\end{proof}

\subsection{Expansion of the normalized measure}

The purpose of the present section is to prove for the measures associated to the empirical $U$-statistics $(\eta^N_t)^{\odot q}$ properties similar to the ones obtained in the unnormalized case. We conclude the section with a Wick formula in this setting.

We set, for $F\in \mathcal{B}_b(E^q)$:

$$
\p_{t,q}^N(F) = \E( (\eta^N_t)^{\odot q} (F))\ .
$$

\begin{theorem}
 \label{Theo:FKnormalized}
The sequence $(\mathbb{P}^N_{t,q})$ is differentiable up to any order.

\end{theorem}

Let us mention that an explicit formula (that we omit) for the derivatives $\partial^l \mathbb{P}_{t,q}$ follows immediately from Fla~\ref{formm} in our proof.

\proof
Let us expand first $(\eta_t^N)^{\odot q }$ in terms of $(\gamma^N_t)^{\odot q}$. Let, in the following, $F$ be a bounded symmetric function of $q$ variables. We first have:
\begin{eqnarray*}
(\eta_t^N)^{\odot q } (F) & = & (\gamma^N_t)^{\odot q} (F) \gamma^N_t(1)^{-q}\\
& = & (\gamma^N_t)^{\odot q} (F) \frac{1}{\gamma_t(1)^q} \frac{1}{\left(1-\left( 1- \frac{\gamma^N_t(1)}{\gamma_t(1)} \right) \right)^q}\ .
\end{eqnarray*}
We set $\check{F}_t:=\frac{F}{\gamma_t(1)^q}$ and get:

$$
(\eta_t^N)^{\odot q } (F) = (\gamma^N_t)^{\odot q} (\check{F}_t) \frac{1}{\left(1-\left( 1- \frac{\gamma^N_t(1)}{\gamma_t(1)} \right) \right)^q}=(\gamma^N_t)^{\odot q} (\check{F}_t)\frac{1}{(1-u^N_t)^q}\ ,
$$
with $$u^N_t=1-\frac{\gamma^N_t(1)}{\gamma_t(1)} \ .$$

Recall the decomposition (that holds for any $m$, see \cite[Lemma 4.11]{DPR-2006}):
\begin{eqnarray*}
 \frac{1}{(1-u^N_t)^q} &=& \sum_{0\leq k \leq m} (q-1+k)_k \frac{(u^N_t)^k}{k!} + (u^N_t)^{m} \sum_{1\leq k \leq q} C_{q+m}^{k+m} \left( \frac{u^N_t}{1-u^N_t}\right)^k
\end{eqnarray*}
from which it follows that:
\begin{eqnarray*}
\nonumber
 (\eta_t^N)^{\odot q } (F) & = & (\gamma^N_t)^{\odot q} (\check F_t) \sum_{0\leq k \leq m} (q-1+k)_k \frac{(u^N_t)^k}{k!}  \\
\nonumber
&&~~~~~ + (\gamma^N_t)^{\odot q} (\check{F}_t)   (u^N_t)^{m} \sum_{1\leq k \leq q} C_{q+m}^{k+m} \left( \frac{u^N_t}{1-u^N_t}\right)^k\\
&&~~~~=: (1)+(2)\ .
\end{eqnarray*}

For an arbitrary function $f$ on $E$, we set: $\theta_t(f):=\frac{f-\eta_t(f)}{\gamma_t(1)}$. Then, recalling the equality $\gamma_t(1) = \exp\left( -\int_0^t \eta_s(V_s) ds \right)$, we get
\begin{eqnarray*}
\frac{\gamma^N_t(1)}{\gamma_t(1)} &=& e^{ -\int_0^t [\eta_u^N(V_u)-\eta_u(V_u)] du }\\
&=&1- \int_0^t [\eta_u^N(V_u)-\eta_u(V_u)] \frac{\gamma^N_u(1)}{\gamma_u(1)}du\\
&=&1-\int_0^t\gamma_u^N(\theta_u(V))du \ ,
\end{eqnarray*}
where the last identity follows from the rewriting $$[\eta_u^N(V_u)-\eta_u(V_u)]\frac{\gamma^N_u(1)}{\gamma_u(1)}=\eta_u^N\left( \frac{V_u-\eta_u(V_u)}{\gamma_u(1)} \right)\gamma^N_u(1)=\gamma_u^N(\theta_u(V_u)).$$

In particular,  we get, for an arbitrary $k>0$, 
\begin{equation}
\label{Eq:decint}
 \left(1-\frac{\gamma^N_t(1)}{\gamma_t(1)}   \right)^k = \int_{s_1,\dots s_k \in [0,t] }\prod_{i=1}^k \gamma_{s_i}^N(\theta_{s_i}(V_{s_i})) ds_1 \dots ds_k
\end{equation}
$$=k!\ \int_{0\leq s_1\leq \dots \leq s_k\leq t  }\prod_{i=1}^k \gamma_{s_i}^N(\theta_{s_i}(V_{s_i})) ds_1 \dots ds_k.$$
We therefore have
\begin{eqnarray}
\nonumber
(1) &=& 
 \sum_{0\leq k \leq m} (q-1+k)_k \int_{0\leq s_1 \leq \dots\leq s_k \leq t} \left[ \prod_{i=1}^k \gamma_{s_i}^N(\theta_{s_i}(V_{s_i}))\right](\gamma^N_t)^{\odot q} (\check{F}_t)  ds_1 \dots ds_k \\
&=&  \sum_{0\leq k \leq m}  (q-1+k)_k \times \int_{0\leq s_1 \leq \dots\leq s_k \leq t} \gamma^{k,q,N}_{{\mathbf s},t}( S^{k,q}_{s_1,\dots,s_k,t}(F) )ds_1\dots ds_k \label{Eq:dev1}
\end{eqnarray} 
with ${\mathbf s}:=\{s_1,\dots,s_k\}$, $\gamma^{k,q,N}_{{\mathbf s},t}= \gamma^N_{s_1}\otimes \dots \gamma^N_{s_k} \otimes (\gamma^N_t)^{\odot q} $, $S^{k,q}_{{\mathbf s},t}(F) = \theta_{s_1}(V_{s_1}) \otimes \dots \otimes \theta_{s_k}(V_{s_k}) \otimes \check{F}_t$. 

We introduce the operators (for $G$ of $k+q$ variables, $j\leq k$)

$$
D_{j}G(x_1,\dots,x_{k+q}) = \frac{N-q-k+j}{N} G(x_{1},\dots,x_{k+q})+ \frac{1}{N}\sum_{j+1\leq i \leq k+q} G_j(x_1,\dots,x_{k+q})(x_i)
$$

We then have, for any empirical measure $m(x)=\frac{1}{N}\sum_{1\leq i \leq N} \delta_{x_i}$ and for any function $G$ of $k+q\leq N$ variables and any measure $\mu$ on $E^{j-1}$
$$
\mu\otimes (m(x) \otimes m(x)^{\odot (k+q-j)}) (G) =\mu \otimes m(x)^{\odot (k+q-j+1)} (D_{j}G)\ .
$$
For $G$ function of $q+k$ variables and $s\leq t$ and $j\leq k+q$, we define then the Markov operator
\begin{eqnarray}
\label{Eq:defRq}
\lefteqn{R^j_{s,t} G(x_1,\dots,x_{k+q}) =}\\
\nonumber &&\E(G(x_1,\dots,x_{k+q-j},\hat{\xi}^{1}_{t-s},\dots,\hat{\xi}^j_{t-s})e^{-\int_0^{t-s}(V_{s+u}(\hat\xi_u^1)+\dots +V_{s+u}(\hat\xi_u^j))du}\\
\nonumber
&& ~~~~~~~~~~~~~~~~~~~~~~~~~|\hat{\xi}^1_0=x_{k+q-j+1},\dots,\hat{\xi}^j_{0}=x_{k+q})\ ,
\end{eqnarray}
where the $\hat\xi^i$ are defined as before, excepted for the initial condition that reads now $\hat\xi_0^i=x_{k+i}$.

Reasoning as before (and using the auxiliary system $\hat\xi_s^i$ as in Sect.~\ref{unnorm}), we get  for $s_k\leq t$ and $H$ a function of $k+q$ variables, $\mu$ a measure on $E^{j-1}$,
\begin{eqnarray*}
\E(\mu\otimes (\gamma^N_{s_k} \otimes (\gamma^N_t)^{\odot q}) (H)|{\cal F}_{s_k})
&=&  \mu \otimes(\gamma^N_{s_k} \otimes (\gamma^N_{s_k})^{\odot q} )R^{q}_{s_k,t} (H)\\
&=& \mu \otimes (\gamma^N_{s_k})^{\odot (q+1)} D_{q+1} R^{q}_{s_k,t} (H)\ ,
\end{eqnarray*}
where ${\cal F}_{t}$ stands for the natural filtration on the probability space underlying the particle system. Here, we took advantage of the Markovian properties of the system and of Eq~\ref{fundam}.
So, by recurrence~:

\begin{eqnarray*}
\E(\gamma^{k,q,N}_{{\mathbf s},t}( G )) = 
(\eta_0)^{\otimes(k+q)}  R^{k+q}_{0,s_1} D_1 R^{k-1+q}_{s_1,s_2} \dots D_k R^{q}_{s_k,t}( G)\ .
\end{eqnarray*}

We proceed now as for the unnormalized measure and introduce still another more general auxiliary particle system $({\check{\xi}}^1_t,\dots,\check{\xi}^{k+q}_t)_{t\geq 0}$. The system depends on $k$ (and $q$), but we do not emphasize this dependency to abbreviate the notations. This should not create ambiguities since the dependency on $k$ should be obvious from the context in the following formulas.

Let $S_1,S_2,\dots,S_k$ be the order statistics of $k$ uniform random variables in $[0,t]$ (we take $k$ i.i.d. uniform variables in $[0,t]$ and sort them). The $S_i$ are independent of all the other variables. We set $S_i:=t$ for $k<i\leq k+q$. The system (which agrees with the auxiliary system $\hat\xi^i_t$ when $k=0$) is defined as follows:
\begin{itemize}
 \item The law of $(\check{\xi}^1_0,\dots,\check{\xi}^{k+q}_0)$ is $\eta_0^{\otimes k+q}$
 \item The particles $\check{\xi}^i_s$ diffuse according to $L_s$ and undergo the following jumps: 
\item Each couple $(i,j)\in [q+k]^2$ has an exponential clock of parameter $V_\infty/N$. At a ringing time $s$:
\begin{itemize}
 \item If  $s>S_i$ or $s>S_j$, then nothing happens.
\item Else,
$$
\check{\xi}^i_s 
\begin{cases}
 \leftarrow \check{\xi}^j_s & \text{ with proba. } \frac{V_s(\check{\xi}^i)}{V_\infty}\\
\leftarrow \check{\xi}^i_s   & \text{ with proba. } 1-\frac{V_s(\check{\xi}^i)}{V_\infty}
\end{cases}
$$
\item At $s=S_i$, we make the following replacement (with $U$ sampled uniformly in $\{i+1,\dots,k+q\}$) 
$$
\check{\xi}^U_s \leftarrow \check{\xi}^i_s \text{ with proba. } \frac{k-i+q}{N}\ .
$$
We say that there has been a parasite at $s$, so that the system has at most $k$ parasites.
\end{itemize}

\end{itemize}

With the notation
\begin{eqnarray*}
\lefteqn{
\check S_{{\mathbf S},t}^{k,q} (\check{\xi}^1_{S_1},\dots,\check{\xi}^k_{S_k},\check{\xi}^{k+1}_t,\dots,\check{\xi}^{k+q}_t)
}
\\
&:=&S_{{\mathbf S},t}^{k,q} (\check{\xi}^1_{S_1},\dots,\check{\xi}^k_{S_k},\check{\xi}^{k+1}_t,\dots,\check{\xi}^{k+q}_t)e^{-\int_0^{S_1}\dots\int_0^t V_u(\check\xi_u^1)du\dots V_u(\check\xi_u^{k+q})du}
\ ,
\end{eqnarray*}
the equation (\ref{Eq:dev1}) gives

\begin{eqnarray*}
\E((1)) & = & \sum_{0\leq k \leq m} \frac{(k-1+q)_k}{k!}\E(\check S_{{\mathbf S},t}^{k,q} (\check{\xi}^1_{S_1},\dots,\check{\xi}^k_{S_k},\check{\xi}^{k+1}_t,\dots,\check{\xi}^{k+q}_t)) \\
&& =\sum_{0\leq k \leq m} \frac{(k-1+q)_k}{k!} \\
&&~~~~~~~~~~\times \E\left\{ \sum_{r=0}^{+\infty}  \frac{\left( \Lambda_S V_\infty/N\right)^re^{-\Lambda_s V_\infty/N}}{r!} \E(\check S_{{\mathbf S},t}^{k,q} (\check{\xi}^1_{S_1},\dots)|r \text{ rings },S_1,\dots,S_k) \right\}
\end{eqnarray*}
with 
$$
\Lambda_S := q(q-1)(t-S_k) +(q+1)q(S_k-S_{k-1})+\dots +(q+k)(q+k-1) (S_1-0)\ ,
$$
for which we will use the crude bound $\Lambda_S \leq (q+k)^2 t$.

We set 
$$
\bar{E}_{k,r,l} = \E( \Lambda_S^l \E(\check S_{{\mathbf S},t}^{k,q} (\check{\xi}^1_{s_1},\dots)|r \text{ rings },S_1,\dots,S_k))\ ,
$$
and notice that the parasites are sampled independently of the $S_i$ and of the ringing times, so that:
$$
\bar{E}_{k,r,l} = \sum_{i=0}^k\p(i~\text{parasites})\E( \Lambda_S^l \E(\check S_{{\mathbf S},t}^{k,q} (\check{\xi}^1_{S_1},\dots)|i ~\text{parasites}, r~\text{rings},S_1,\dots,S_k))\ .$$
The probability $\p(i~ \text{parasites})$ depends on $t$, $q$ and $k$, but we omit the corresponding indices to simplify the notation.
The following lemma is crucial:
\begin{lemma}\label{crucial}
For $i+r<\frac{k}{2}$, we have:
$$\bar{D}_{k,r,l,i}:=\E( \Lambda_S^l \E(\check S_{{\mathbf S},t}^{k,q} (\check{\xi}^1_{s_1},\dots)|i ~\text{parasites}, r~\text{rings},S_1,\dots,S_k))=0\ ,$$ so that:
$$
\bar{E}_{k,r,l} = \sum_{i=(\left\lceil {\frac{k}{2}}\right\rceil-r)^+}^k\p(i~ \text{parasites})\bar{D}_{k,r,l,i}\ .$$
where $\left\lceil x\right\rceil$ stands for $\inf\{k\in \N:k\geq x\}$.
\end{lemma}

Indeed, for $i+r<\frac{k}{2}$, at least one of the $\check{\xi}^j$, $j\leq k$ does not interact with the other particles, so that, by Fubini's lemma, $\theta_{S_j}(V_{S_j})(\check{\xi}^j_{S_j})e^{-\int_0^{S_j}V_u(\check\xi_u^j)du}$ can be integrated independently. 
However, conditionnally to the assumption that $\check{\xi}^j$ does not interact with the other particles, we have:
$$\E (\theta_{S_j}(V_{S_j})(\check\xi_{S_j}^j)e^{-\int_0^{S_j}V_u(\check\xi_u^j)du})=\eta_{S_j}(\theta_{S_j}(V_{S_j}))=0,$$
and the lemma follows.

We then have (by developping the exponential in  $\E((1))$)

\begin{eqnarray}
\label{Eq:esp(1)}
\lefteqn{
\E((1)) = \sum_{0\leq k \leq m} \frac{(k+q-1)_k}{k!} \sum_{r=0}^{+\infty} \frac{V_\infty^r }{N^r r!} \sum_{i=0}^{+\infty} \frac{(-1)^i}{i!} \left(\frac{V_\infty }{N} \right)^i \bar{E}_{k,r,r+i}
}
\\
\label{formm} 
&=&\sum_{0\leq k \leq m} \frac{(k+q-1)_k}{k!} \sum_{r=0}^{+\infty} \frac{V_\infty^r }{N^r r!} \sum_{i=0}^{+\infty} \frac{(-1)^i}{i!} \left(\frac{V_\infty }{N} \right)^i \\
\nonumber
&&~~~~~~~~~~~~~ \sum_{j=(\left\lceil \frac{k}{2}\right\rceil-r)^+}^k\p(j~ \text{parasites})\bar{D}_{k,r,r+i,j}\ . 
\end{eqnarray}
By the construction of the auxiliary system of particles, $\p(j~ \text{parasites})=O(\frac{1}{N^j})$.
For further use, the explicit formula is:
\begin{eqnarray*}
\lefteqn{\p(j \text{ parasites })}\\
& =&
 \sum_{1\leq k_1 \leq \dots \leq k_j \leq k} \left( \frac{(q+k_1-1)}{N} \dots \frac{q+k_j-1}{N}  \prod_{l\in [k] \backslash \{k_1,\dots,k_j\}} \left( 1-\frac{q+l-1}{N} \right)\right) \ ,
\end{eqnarray*}
so that there exists coefficients $(a_{l})$ (depending on $k$ and $q$, as usual we omit the corresponding indices for notational simplicity) such that
$$
\p(j \text{ parasites }) = \sum_{l=j}^k \frac{a_{l}}{N^l}\ .
$$
Because of the condition $i+r<\frac{k}{2}$ in Lemma~\ref{crucial}, it follows that the coefficient of any $\frac{1}{N^p}$ in the expansion of $\E((1))$ is a finite sum of coefficients which are independent of $N$. This is essentially all what we need in order to prove the Thm.~\ref{Theo:FKnormalized}, provided we are able (Step 1 below) to give a suitable upper bound for the remainder terms in the expansion of $\E((1))$ at a given order, and (Step 2) to give a suitable upper bound to $\E((2))$. This is the (cumbersome) purpose of the following conclusion to the proof of the Thm.~\ref{Theo:FKnormalized}, that we include for completeness sake but that the reader may skip, if she/he wishes.

\underline{Step 1}: Upper bound for the remainder terms in the expansion of $E((1))$ at the order p, assuming that $m$ is even and $m=2p$.

Notice that for an arbitrary bounded function $F$, 
$$
\Vert \check F \Vert_\infty \leq  \Vert F \Vert_\infty e^{tV_\infty q}\ ,
$$
$$
\Vert S_{{\mathbf S},t}^{k,q}\Vert_\infty \leq 2^{k} V_\infty^{k} e^{t(k+q)V_\infty} \Vert F \Vert_\infty \ ,
$$
and so
$$
|\bar{E}_{k,r,r+i}| \leq (t (q+k)^2)^{r+i} 2^k V_\infty^k e^{t(k+q)V_\infty} \Vert F\Vert_\infty\ .
$$
  In the Fla~\ref{Eq:esp(1)} for the expectation $\E((1))$,  the sum of the terms with a factor $1/N^{r+i}$, $r+i\geq p$ is, in absolute value, equal to (we use here: $\forall x\geq 0$, $\sum_{i\geq j} \frac{x^i}{i!} \leq \frac{x^{j}}{j!} e^x$)
\begin{eqnarray*}
&& \left|    \sum_{0\leq k \leq 2p}   \frac{(k+q-1)_k}{k!}  \sum_{r=0}^{+\infty} \frac{V_\infty^r }{N^r r!} \sum_{i\geq (p-r)_+} \frac{(-1)^i}{i!} \left(\frac{V_\infty }{N} \right)^i \bar{E}_{k,r,r+i}  \right|\\
&&~~~~\leq   \sum_{0\leq k \leq 2p} \frac{(k+q-1)_k}{k!} \left[ \sum_{r=0}^{p-1} \frac{V_\infty^r}{N^r r!} 
2^{k} V_\infty^{k} e^{t(k+q)V_\infty} \Vert F \Vert_\infty \right.\\
&&~~~~ \times  \frac{1}{(p-r)!} \left( \frac{V_\infty t(q+k)^2}{N}\right)^{p-r} \exp\left(\frac{V_\infty t(q+k)^2}{N} \right) (t(k+q)^2)^r  \\
&&~~~~~~~~+ \left.  \sum_{r=p}^{+\infty} \frac{V_\infty^r}{N^r r!} 2^{k} V_\infty^{k} e^{t(k+q) V_\infty} \Vert F \Vert_\infty \exp\left(\frac{V_\infty t(q+k)^2}{N} \right) (t(q+k)^2)^r \right]\\
&&~~~~~ \leq \frac{1}{N^p} \sum_{0\leq k \leq 2p} \frac{(k+q-1)_k}{k!} \\
&&~~~~~~~~~~\times \left[   \sum_{r=0}^{p-1} V_\infty^{p+k} 2^k e^{t(k+q)V_\infty} \Vert F \Vert_\infty \frac{1}{(p-r)!r!} (t(q+k)^2)^p e^{\frac{V_\infty t(q+k)^2}{N}} \right. \\
&&~~~~~~~~~ \left. +  2^k \frac{(t(k+q)^2)^pV_\infty^{k+p}}{p!}e^{V_\infty t (k+q)} \Vert F \Vert_\infty e^{2\frac{V_\infty t (q+k)^2}{N}}   \right]=O(\frac{1}{N^p})\ .
\end{eqnarray*}

Since, by Fla~\ref{formm}, $\E((1))$ is the sum of this term with a polynomial in $\frac{1}{N}$,  the proof of the Step 1 is concluded.

\underline{Step 2}: Upper bound to $E((2))$.

Since $\gamma_t(1)=e^{-\int_0^t\eta_u(V_u)du}$ and $\gamma_t^N(1)=e^{-\int_0^t\eta_u^N(V_u)du}$, we get:
$$
|\E((2))| \leq  \Vert F \Vert_\infty e^{tq V_\infty} \left( \sum_{1\leq k \leq q} C_{q+2p}^{k+2p} (e^{tV_\infty}+1)^k\right) \E((u^N_t)^{2p})\ .
$$
We have, with the same auxiliary system of the $\check{\xi}$'s as before but with $q=0$, $k=2p$ and with $S_1,\dots,S_{2p}$ the order statistics of $2p$ uniform variables on $[0,t]$,
\begin{eqnarray*}
\lefteqn{ \E((u^N_t)^{2p})}\\ & = & \E(\theta_{S_1}(V_{S_1})\otimes \dots\otimes \theta_{S_{2p}}(V_{S_{2p}}) (\check{\xi}^1_{S_1},\dots,\check{\xi}^{2p}_{S_{2p}})e^{-\int_0^{S_1}\dots\int_0^{S_{2p}}V_{u_1}(\check\xi_{u_1}^1)\dots V_{u_p}(\check\xi_{u_{2p}}^{2p})du_1\dots du_{2p}})\ .
\end{eqnarray*}
We get (with, now, $\Lambda_S= (2p)(2p-1)S_1+(2p-1)(2p-2)(S_2-S_1)+\dots+1\times(S_{2p}-S_{2p-1})$)

\begin{eqnarray*}
\E((u^N_t)^{2p})&&= \E(\sum_{r=0}^{+\infty} \left( \frac{\Lambda_S V_\infty}{N} \right)^r \frac{1}{r!} e^{-\Lambda_S V_\infty/N} \E(  \theta_{S_1}(V_{S_1})\otimes \dots\otimes \theta_{S_{2p}}(V_{S_{2p}})\\
&& (\check{\xi}^1_{S_1},\dots,\check{\xi}^{2p}_{S_{2p}})e^{-\int_0^{S_1}\dots\int_0^{S_{2p}}V_{u_1}(\check\xi_{u_1}^1)\dots V_{u_p}(\check\xi_{u_{2p}}^{2p})du_1\dots du_{2p}}|r \text{ rings},S_1,\dots,S_{2p}))\\
&& =   \E(\sum_{r=0}^{+\infty} \left( \frac{\Lambda_S V_\infty}{N} \right)^r \frac{1}{r!} e^{-\Lambda_S V_\infty/N}   \sum_{i=0}^{2p-1}\p(i \text{ parasites})\\
&&~~~~~~~ \times \E(  \theta_{S_1}(V_{S_1})\otimes \dots\otimes \theta_{S_{2p}}(V_{S_{2p}}) (\check{\xi}^1_{S_1},\dots,\check{\xi}^{2p}_{S_{2p}})\\
&&e^{-\int_0^{S_1}\dots\int_0^{S_{2p}}V_{u_1}(\check\xi_{u_1}^1)\dots V_{u_p}(\check\xi_{u_{2p}}^{2p})du_1\dots du_{2p}}|r \text{ rings}, i \text{ parasites},S_1,\dots,S_{2p}))\ .
\end{eqnarray*}
We have 
$$
\p( i \text{ parasites}) \leq C_{2p-1}^i \left(  \frac{2p-1}{N} \right)^i
$$
and know that, if $r+i<p$
$$
\E(  \theta_{S_1}(V_{S_1})\otimes \dots\otimes \theta_{S_{2p}}(V_{S_{2p}}) (\check{\xi}^1_{S_1},\dots,\check{\xi}^{2p}_{S_{2p}})$$
$$e^{-\int_0^{S_1}\dots\int_0^{S_{2p}}V_{u_1}(\check\xi_{u_1}^1)\dots V_{u_p}(\check\xi_{u_{2p}}^{2p})du_1\dots du_{2p}}|r \text{ rings}, i \text{ parasites},S_1,\dots,S_{2p})) = 0 \ .
$$
So that
\begin{eqnarray*}
 \E((u^N_t)^{2p}) & \leq & \frac{1}{N^p} \sum_{r=0}^{+\infty} \frac{(t(2p)^2 V_\infty)^r}{r!} \sum_{i=(p-r)_+}^{2p-1} C_{2p-1}^i ( 2p-1)^i e^{2p t V_\infty}2^{2p}V_\infty^{2p}
 \end{eqnarray*}
 and
 \begin{eqnarray*}
|\E((2))|& \leq &  \frac{1}{N^p} \times \Vert F \Vert_\infty e^{tq V_\infty} \left(\sum_{1\leq k \leq q} C_{q+2p}^{k+2p} (e^{tV_\infty}+1)^k \right) \\
&& ~~ \times  e^{t(2p)^2 V_\infty}  \sum_{i=0}^{2p-1} C_{2p-1}^i ( 2p-1)^i e^{2p t V_\infty}2^{2p}V_\infty^{2p}=O(\frac{1}{N^p}),
\end{eqnarray*}
and the proof of the Thm is complete.
\endproof

\begin{corollary}
\label{Cor:FKWick}
 If $F\in \Bsym (E^q)$, then 
\begin{itemize}
 \item  $\partial^l \mathbb{P}_{t,q} (F)=0$, $\forall l <q/2$ 
\item for $q$ even,  $\partial^{q/2} \p_{t,q}(F)$ simplifies to
$$
\partial^{q/2} \p_{t,q}(F)=
\frac{V_\infty^{q/2}}{(q/2)!} (q(q-1) t)^{q/2} \E(\check F^e(\hat{\xi}^1_t,\dots,\hat{\xi}^q_t)|\frac{q}{2} \text{rings})\ ,
$$
or
$$\partial^{q/2} \p_{t,q}(F)=\frac{\partial^{q/2} \Q_{t,q}(F)}{\gamma_t(1)^q}.$$
\end{itemize}

\end{corollary}

\proof[Proof of Corollary \ref{Cor:FKWick}]
Indeed, the same argument as in the proof of Lemma~\ref{crucial} shows that, for a centered function $F\in{\cal B}_0^{sym}$, the terms $\overline{D}_{k,r,r+i,j}$ vanish for $j+r<\frac{k+q}{2}$ (this is because at least one of the $\check\xi^j,\ j\leq k+q$ does not interact with the other particles if there are less than $\frac{k+q}{2}$ rings and parasites). 
It follows that, if $q$ is even (resp. odd), $\frac{q}{2}$ (resp. $\frac{q+1}{2}$) is the smallest power of $\frac{1}{N}$ appearing in the expansion of $\E((1))$ for $m>q$ and that the contribution of $(2)$ is negligible with respect to $1/N^{q/2}$, from which the first part of the Corollary follows.

The same reasoning shows that, for $q$ even, only the case $k=i=j=0$ and $r=\frac{q}{2}$ contributes to the coefficient of $(\frac{1}{N})^{\frac{q}{2}}$. As noticed before, for $k=0$, the systems $(\check{\xi})$ and $(\hat{\xi})$ have the same law.
 The second statement follows.

 \section{First results on asymptotic normality}

The empirical measure $\eta_t^N$ converges weakly to $\eta_t$: for all bounded $f$, $\eta^N_t(f) \underset{N \rightarrow + \infty }{\overset{\text{a.s.}}{\longrightarrow}} \eta_t(f)$. The Wick theorem (Cor. \ref{Cor:FKWick}) allowed us to improve on this convergence result. Namely, we know from the previous section that
 \begin{itemize}
 \item For $q$ odd and $F \in \mathcal{B}^{\text{sym}}_0(E^q)$: $N^{q/2} \E((\eta^N_t )^{\odot q} (F) ) \underset{N \rightarrow \infty}{\longrightarrow} 0$,
 \item For $q$ even, $F \in \mathcal{B}^{\text{sym}}_0(E^q)$: $N^{q/2} \E((\eta^N_t )^{\odot q} (F) ) \underset{N \rightarrow \infty}{\longrightarrow} \Delta_{q/2}(F)$,
 \end{itemize} 
where  $\Delta_{q/2}$ is a shortcut for the signed measure:
$$\Delta_{q/2}(F)=\frac{(q(q-1)tV_\infty)^{q/2}}{(q/2)!} \E(\check F^e(\hat{\xi}^1_t,\dots,\hat{\xi}^q_t)|\frac{q}{2} \text{rings})\ .$$
Furthermore, for $F=(f_1\otimes ...\otimes f_q)_{sym}$, with $f_i\in \mathcal{B}_0 (E)$, this formula simplifies to:
$$\Delta_{q/2}(F)=\frac{(2V_\infty t)^{q/2}}{\gamma_t(1)^q}\sum\limits_{I_q\in{\cal I}_q}\prod\limits_{\{i,j\}\in I_q}E_{t,1}'(f_if_j)=\sum\limits_{I_q\in{\cal I}_q}\prod\limits_{\{i,j\}\in I_q}W_t(f_i \otimes f_j),$$
where $W_t:=\frac{2V_\infty t}{\gamma_t(1)^2}E_{t,1}'$.

A striking consequence of this Wick theorem for interacting particle systems and their U-statistics is that it leads immediately to an explicit result of asymptotic normality for the vector-valued random measures $(\eta_t^N)^q$. 
Notice that the same properties hold for the system studied in \cite{rubenthaler-2009} and so we have the same asymptotic normality result for this system.
More precisely:
 
 \begin{theorem} \label{Theo:CLT}
 $\forall f_1,\dots,f_q \in \mathcal{B}_0 (E)$, 
 $$
 N^{1/2} (\eta^N_t(f_1),\dots,\eta^N_t(f_q)) \underset{N\rightarrow +\infty}{\overset{\text{law}}{\longrightarrow}} \mathcal{N}(0,K)
 $$
 with $K(i,j)= \eta_t(f_i f_j) + \E(W_t(f_i \otimes f_j)) $.
 
 \end{theorem}
 
 \begin{proof}
 For any $u_1,\dots,u_q$, we have:
 \begin{eqnarray*}
&& \E\left(\exp\left( N \eta^N_t \left(\log\left(1+\frac{iu_1 f_1 + \dots +iu_q f_q}{\sqrt{N}} \right)\right)  \right)\right)  \\
&& =  \E\left( \exp\left( \sum_{k\geq 1} \frac{(-1)^{k+1}}{k} N^{1-k/2} \eta^N_t( (iu_1 f_1 + \dots iu_q f_q )^k )\right)  \right)\\
&&  \underset{N \rightarrow +\infty}{\sim}    \E(\exp(\sqrt{N} (i u_1 \eta^N_t(f_1) + \dots + i u_q \eta^N_t (f_q)))) \\ 
&&~~~~~~~~~~~\times   \exp\left( -\frac{1}{2}  \eta_t ((u_1f_1 + \dots + u_q f_q)^2) ) \right) \ ,
 \end{eqnarray*}
where the last equivalence follows from the differentiability of $\p_{t,1}^N$.

We also have
 \begin{eqnarray*}
 && \E\left(\exp\left( N \eta^N_t \left(\log\left(1+\frac{iu_1 f_1 + \dots iu_q f_q}{\sqrt{N}} \right)\right)  \right)\right)  \\
 && ~~~ = \E\left(  \prod_{j=1}^N \left(  1+ \frac{ i u_1 f_1(\xi^j_t) + \dots + i u_q f_q (\xi^j_t)}{\sqrt{N}}\right)\right) \\ 
 && ~~~ = \E\left(\sum_{0 \leq k \leq N} \frac{1}{N^{k/2}} \sum_{1\leq j_1 ,\dots,j_k \leq q } i^k  u_{j_1} \dots u_{j_k} \sum_{1\leq i_1 < \dots < i_k \leq N } f_{j_1}(\xi^{i_1}_t) \dots f_{j_k}(\xi^{i_k}_t)\right)\\
 && ~~~ = \E\left(\sum_{0 \leq k \leq N} \frac{(N)_k}{N^{k/2}}  \sum_{1\leq j_1, \dots , j_k \leq q } i^k  u_{j_1} \dots u_{j_k}  \frac{1}{k!}   ( \eta^N_t )^{\odot k} (f_{j_1} \otimes \dots \otimes f_{j_k} )\right) \\
 && ~~~ \underset{N\rightarrow + \infty}{\longrightarrow}  \sum_{k \geq 0, k \text{ even }} (-1)^{k/2}  \sum_{1\leq j_1, \dots , j_k \leq q}\frac{ u_{j_1} \dots u_{j_k}}{k!}  \sum\limits_{I_k\in{\cal I}_k}\prod\limits_{\{a,b\}\in I_k}W_t(f_a \otimes f_b)\\
 &&~~~=\sum_{k \geq 0, k \text{ even }} \frac{(-1)^{k/2}}{2^{k/2} (k/2)!}  \sum_{1\leq j_1, \dots , j_k \leq q} u_{j_1} \dots u_{j_k}  \E(W_t(f_{j_1} \otimes f_{j_2})) \dots \E(W_t(f_{j_{k-1}} \otimes f_{j_k})) \\
 && ~~~ = \sum_{k \geq 0, k \text{ even }} \frac{(-1)^{k/2}}{2^{k/2} (k/2)!} \left( \sum_{1\leq j_1,j_2 \leq q} u_{j_1} u_{j_2} \E( W_t(f_{j_1}\otimes f_{j_2})) \right)^{k/2} \\
 && ~~~ = \exp\left(  -  \frac{1}{2} \sum_{ 1\leq j_1,j_2 \leq q} u_{j_1} u_{j_2} \E(W_t(f_{j_1}\otimes f_{j_2})) \right)
 \end{eqnarray*}
 The Theorem follows.
 \end{proof}

 \section{Convergence of empirical $U$-statistics.}

\subsection{Hoeffding's decomposition}

We want here to study the convergence under fairly general assumptions of a sequence of empirical U-statistics $(\eta_t^N)^{\odot q}(F)$ when $N\rightarrow +\infty$, where the $\eta_t^N$s are empirical measures on some space $E$ and $F $ is a bounded symmetrical function on $E^q$. 

Let us first write an analog of Hoeffding's decomposition (cf. \cite{lee-1990}, \cite{delapena-gine-1999}). We fix some measure $\eta_t$. 
For $1\leq k \leq q$ and any $F$ bounded symmetrical on  $E^q$, we define

$$
F^{(k)}(x_1,\dots,x_q)= \int_{E^{q-k}} F(x_1,\dots,x_k,x_{k+1},\dots,x_q) \eta_t^{\otimes (q-k)}(dx_{k+1},\dots,dx_{q})\ .
$$
Notice that $F^{(q)}=F$ and that, for an arbitrary $k$, $F^{(k)}$ is symmetrical on $E^k$.
We define recursively (the function $F$ being fixed --for notational simplicity, we do not stress the dependency of $\theta$, $h^{(i)}$... in the following formulas):

\begin{eqnarray*}
 \theta &=& \int_{E_n^q} F(x_1,\dots,x_q) \eta_t^{\otimes q}(dx_1,\dots,dx_q)\\
h^{(1)}(x_1) &=& F^{(1)}(x_1)-\theta\\
h^{(k)}(x_1,\dots,x_k) &=& F^{(k)}(x_1,\dots,x_k) - \sum_{j=1}^{k-1} \sum_{(k,j)}h^{(j)} -\theta\ ,
\end{eqnarray*}
where $\sum_{(k,j)}h^{(j)}$ is an abbreviation for
$\sum_{1\leq i_1 < \dots < i_j \leq k}[ h^{(j)}(x_{i_1},\dots,x_{i_j})]$.

The $h^{(k)}$ are bounded and symmetrical. They are constructed so as to be centered and to satisfy, for $l<k$, $(h^{(k)})_l=0$, see e.g. \cite[Sect. 1.6]{lee-1990}.

\subsection{Asymptotic normality for empirical U-statistics.}

As for classical U-statistics, the Hoeffding decomposition can be used to study the convergence of the empirical U-statistics $(\eta_t^N)^{\odot q}$.
Notice first that, $\forall N$,

\begin{equation}
 \label{Eq:dec1}
 (\eta_t^N)^{\odot q} (F) = (\eta_t^N)^{\odot q} (h^{(q)}) +(\eta_t^N)^{\odot q}(\sum_{j=1}^{q-1} \sum_{(q,j)} (h^{(j)})) +\theta\ .
\end{equation}
Notice that, $\forall 1\leq j \leq q-1$,
\begin{eqnarray*}
 \sum_{(q,j)} h^{(j)}& = & \frac{1}{j!} \sum_{b\in \langle j,q \rangle } h^{(j)} (x_{b(1)},\dots,x_{b(j)})\ .
\end{eqnarray*}
So, with the notations of the first sections:
\begin{eqnarray} 
\nonumber
 (\eta_t^N)^{\odot q} (\sum_{(q,j)} h^{(j)}) &=& \frac{1}{(N)_q} \sum_{a \in \langle q,N \rangle} \frac{1}{j!} \sum_{b\in \langle j,q \rangle} h^{(j)} (\xi^{a\circ b (1)}_t,\dots,\xi^{a\circ b (j)}_t)\\
\nonumber
&=& \frac{1}{ j!(N)_q} \times (N-j)_{(q-j)} (q)_j \sum_{a' \in <j,N>} h^{(j)}(\xi^{a'(1)}_t,\dots,\xi^{a'(j)}_t)\\
\label{Eq:stat1}
&=& \frac{(q)_j}{j!} (\eta_t^N)^{\odot j}(h^{(j)})\ .
\end{eqnarray}

\begin{theorem}
Suppose we have, for all $ j \geq 2$ and $f \in \mathcal{B}^{\text{sym}}_0 (E^j)$,
\begin{equation}
\label{Eq:stat2}
\E( ((\eta_t^N)^{\odot j}(f))^2) \leq \frac{C}{N^{j}}\ ,
\end{equation}
for some constant $C$ depending on $t,N,j,\Vert f \Vert_\infty$. Assume further that, with $F$ as above,
\begin{equation}
\label{Eq:stat3}
\sqrt{N} \eta_t^N(h^{(1)}) = \sqrt{N}(\eta_t^N(F^{(1)}) - \theta) \overset{\text{law}}{\underset{N\rightarrow +\infty}{\longrightarrow}} \mathcal{N}(0,\sigma^2)\ ,
\end{equation}
then 
\begin{equation}
\label{Eq:stat4}
\sqrt{N}((\eta_t^N)^{\odot q}(F)- \theta)  \overset{\text{law}}{\underset{N\rightarrow +\infty}{\longrightarrow}} \mathcal{N}(0,q^2 \sigma^2)\ .
\end{equation}
\end{theorem}

The Theorem follows from Hoeffding's decomposition, together with the identity 
$$(\eta_t^N)^{\odot q}(\sum\limits_{(q,1)}h^{(1)})=q\ \eta_t^N(h^{(1)}).$$

\begin{corollary}
For the Feynman-Kac particle systems of the first sections, we have:

$$\sqrt{N}((\eta_t^N)^{\odot q}(F)-\theta )\longrightarrow \mathcal{N}(0,q^2\eta_t((F^{(1)})^2)).$$
\end{corollary}

Indeed, notice first that the Eq~\ref{Eq:stat3} holds for Feynman-Kac particle systems e.g. as a Corollary of Th.~\ref{Theo:CLT}. The Corollary follows then from the following Lemma.

\begin{lemma}
For the particle system described in Section \ref{Sec:1}, we have, $\forall j$, $\forall  f \in \mathcal{B}^{sym}_0(E^j)$, 

$$
\E(((\eta^N_t)^{\odot j}(f))^2) \leq \frac{C}{N^j}\ ,
$$
for some constant $C$ (depending on $t$, $\Vert f \Vert_\infty$).
\end{lemma}

\begin{proof}
\begin{eqnarray*}
((\eta^N_t)^{\odot j}(f))^2 & = & \frac{1}{((N)_j)^2} \sum_{a,b \in \langle j,N \rangle } f(\xi^{a(1)}_t,\dots,\xi^{a(j)}_t)f(\xi^{b(1)}_t,\dots,\xi^{b(j)}_t) \\
& = & \frac{1}{((N)_j)^2} \sum_{k=0}^j  \underset{ \# \text{Im}(a) \cap \text{Im}(b) = k}{\sum_{a,b \in \langle j,N \rangle }}f(\xi^{a(1)}_t,\dots,\xi^{a(j)}_t)f(\xi^{b(1)}_t,\dots,\xi^{b(j)}_t)\\
& = & \frac{1}{((N)_j)^2} \sum_{k=0}^j \sum_{a \in \langle 2j-k, N \rangle } C_j^k f(\xi_t^{a(1)},\dots,\xi_t^{a(j)}) \\
&& ~~~~~~~~~~~~~~~~~~~ \times f(\xi_t^{a(1)},\dots, \xi_t^{a(k)},\xi_t^{a(j+1)},\dots,\xi_t^{a(2j-k)}) \\
& = & \frac{(N)_{2j}}{((N)_j)^2} (\eta^N_{t})^{\odot 2j} (f\otimes f) \\
&&\quad + \sum_{k=1}^j \frac{C_j^k (N)_{2j-k}}{((N)_j)^2} \frac{1}{(N)_{2j-k}} \sum_{ a \in \langle 2j-k, N \rangle } f(\xi_t^{a(1)},\dots,\xi_t^{a(j)})\\
&& \qquad \qquad \qquad \qquad  \times f(\xi_t^{a(1)},\dots, \xi_t^{a(k)},\xi_t^{a(j+1)},\dots,\xi_t^{a(2j-k)})
\end{eqnarray*}

Notice first that, because of the very definition of $(\eta^N_{t})^{\odot 2j}$ (as an average over arbitrary configurations of distinct systems of $2j$ particles), although $f\otimes f$ is not a symmetrical function, $(\eta^N_{t})^{\odot 2j} (f\otimes f)=(\eta^N_{t})^{\odot 2j}(Sym (f\otimes f))$, where $Sym$ stands for the canonical symmetrization map from ${\mathcal B}^{sym}_0(E^q)\otimes {\mathcal B}^{sym}_0(E^q)$ to ${\mathcal B}^{sym}_0(E^{2q})$.
Therefore, by Corollary \ref{Cor:FKWick}, $\exists C_0, \frac{(N)_{2j}}{((N)_j)^2} \E((\eta^N_{t})^{\odot 2j} (f\otimes f) ) \leq \frac{C_0}{N^j}$. For $k \in \{1,\dots,j\}$, we set $g_k(x_1,\dots,x_{2j-k})=f(x_1,\dots,x_j)f(x_1,\dots,x_k,x_{j+1},\dots,x_{2j-k})$. We then have
\begin{eqnarray*}
&&\frac{1}{(N)_{2j-k}} \sum_{ a \in \langle 2j-k, N \rangle } f(\xi_t^{a(1)},\dots,\xi_t^{a(j)})  \times f(\xi_t^{a(1)},\dots, \xi_t^{a(k)},\xi_t^{a(j+1)},\dots,\xi_t^{a(2j-k)})  \\
&& \quad  = (\eta^N_t)^{\odot (2j-k)} (g_k) \ .
\end{eqnarray*}
This function $g_k$ has the particular feature that there are $\alpha=2j-2k$ indexes $i$ such that $\int_{E^{2j-k}} g_k(x_1,\dots,x_{2j-k}) \eta_t(dx_i) =0$ (namely $i=k+1,\dots,2j-k$).
As for the proof of Corollary \ref{Cor:FKWick}, we go back to the proof of Lemma \ref{crucial}. For the function $g_k$ defined above, the terms $\overline{D}_{k',r,r+i,j'}$ vanish for $j'+r < \frac{k'+\alpha}{2}$. So $j-k$ is the smallest power of $N$ appearing in the Laurent-type expansion of $\E((\eta^N_t)^{\odot (2j-k)} (g_k))$ and $\exists C_k, \E((\eta^N_t)^{\odot (2j-k)} (g_k)) \leq \frac{C_k}{N^{j-k}}$. This finishes the proof.
\end{proof}

\subsection{Wiener integrals expansions}
The convergence of symmetric statistics to multiple Wiener integrals has been studied intensively, see e.g. the seminal \cite{dynkin-mandelbaum-1983} or the account of the classical theory in \cite{lee-1990}.
We are interested here in proving that the main statements of the theory still hold for empirical U-statistics and focus on a particular example, namely the one of a symmetrical and bounded kernel constructed as a product of centered functions.

From the technical point of view, the key issue w.r. to the classical theory of limit distributions for U-statistics based on i.i.d. assumptions is related to the existence of interactions between the particles of the system. Taking these interactions into account amounts in practice to replace the central limit theorem by the analogous statement for interacting particle systems, namely Thm.~\ref{Theo:CLT}.

Let us consider a kernel (a symmetrical bounded fonction on $E^q$) $F$ of the form $F=(f_1\otimes ...\otimes f_q)_{sym}$ with $f_i\in {\mathcal B}_0(E)$. Then, $F^{(k)}=h^{(k)}=0$ for $k=1,...,q-1$, and we expect $N^{\frac{q}{2}}(\eta_t^N)^{\odot q}(F)$ to converge in law.
This is indeed the case.

Let us start with some classical results (see e.g. \cite{rota-wallstrom-1997} for a systematical approach by means of the poset of partitions).
For an empirical signed measure $m(x)=\sum\limits_{1\leq i\leq N}\delta_{x_i}$, we first have the formula of Rubin and Vitale (\cite{dynkin-mandelbaum-1983},\cite{lee-1990} p. 85):
$$m(x)^{\odot q}(F)=\sum\limits_{\cal P}\prod\limits_{V\in \cal P}(-1)^{|V|-1}(|V|-1)!m(x)(f_V)$$
where $f_V, V=\{v_1,...,v_k\}$, stands for $f_{v_1}...f_{v_k}$ and $\cal P$ runs over the set of partitions of $[q]$ into disjoint subsets.

Assume now that $m(x)=N\cdot \eta_t^N=:m_N$.
For $|V|=1$, $N^{-\frac{1}{2}}m_N(f_V)$ converges in law to ${\cal N}(0,\eta_t(f_V^2)+\E (W_t(f_V \otimes f_V))$. For $|V|\geq 2$, and since $\eta_t^N$ converges to $\eta_t$,  $N^{-1}m_N(f_V)$ converges to $\eta_t(f_V)$. 

Let us consider now a partition $\cal P$ of $[q]$ with $j_1$ sets with 1 element, $j_2$ sets with 2 elements, ..., $j_q$ sets with $q$ elements. We notice that $j_1+2j_2+2j_3+...+2j_p<j_1+2j_2+...+qj_q=q$ excepted if $j_3=...=j_q=0$ (and then the two terms are equal). In particular,
$N^{-\frac{q}{2}}\prod\limits_{V\in\cal P }(-1)^{|V|-1}(|V|-1)!m_N(f_V)$ converges in probability to $0$ if $j_3+...+j_q>0$.
By Slutsky's theorem
it follows that, if $N^{-\frac{q}{2}} m_N^{\odot q}(F)=N^{\frac{q}{2}} (\eta_t^N)^{\odot q}(F)$ converges in law, it converges to the same limit as $\sum\limits_{\overline{\cal P}}\prod\limits_{V\in \overline{\cal P}}(-1)^{|V|-1}(|V|-1)!m_N(f_V)$, where the sum is restricted now to partitions $\overline{\cal P}$ of $[q]$ with $j_3=...=j_q=0$.
The next theorem follows.

\begin{theorem}
With the above asumptions, we have the convergence in law:
$$N^{\frac{q}{2}} (\eta_t^N)^{\odot q}(F)\rightarrow \sum\limits_{k=0}^{\left\lceil \frac{q}{2}\right\rceil}(-1)^k
\sum\limits_{1\leq i_1<...<i_{q-2k}\leq q}I(f_{i_1})...I(f_{i_{q-2k}})\sum\limits_{J}\eta_t(f_{J_1})....\eta_t(f_{J_k}).$$
Here, $J$ runs over the partitions of $[q]-\{i_1,...,i_{q-2k}\}$ in ordered pairs $J_1=\{j_1,j_1'\},...,J_k=\{j_k,j_k'\}$ (where $j_i<j_i'$ and $j_1<...<j_k$).
Besides, $\left\lceil x\right\rceil$ stands for the integer part of a real number $x$ (the highest integer less or equal to $x$).
The $I(f_i), i=1,...,q$ are Wiener integrals of the functions $f_i$. They form a Gaussian family with $\E[I(f_i)]=0$ and $\E[I(f_i)I(f_j)]=\eta_t(f_if_j)+\E(W_t(f_i \otimes  f_j))$.
\end{theorem}

\begin{corollary}
In the particular case $F=f^{\otimes q}$ with $\eta_t(f^2)=1$, we get the convergence in law:
$$N^{\frac{q}{2}} (\eta_t^N)^{\odot q}(F)\rightarrow H_k(I(f))$$
where $H_k$ stands for the $k$-th Hermite polynomial.
\end{corollary}

Indeed, in that case, since the number of partitions of $[q]$ with $q-2k$ singletons and $k$ pairs is
$\frac{q!}{2^k(q-2k)!k!}$, the Theorem simplifies to a convergence in law of $N^{-\frac{q}{2}} (\eta_t^N)^{\odot q}(F)$ to
$$\sum\limits_{k=0}^{\left\lceil \frac{q}{2}\right\rceil}(-1)^k\frac{q!}{2^k(q-2k)!k!}
I(f)^{q-2k},$$
where one recognizes the expansion of the $q$th (``probabilistic'') Hermite polynomial:
$$H_q(x)=\sum\limits_{k=0}^{\left\lceil \frac{q}{2}\right\rceil}(-1)^k\frac{q!}{2^k(q-2k)!k!}
x^{q-2k}.$$

\providecommand{\noopsort}[1]{}
\providecommand{\bysame}{\leavevmode\hbox to3em{\hrulefill}\thinspace}
\providecommand{\MR}{\relax\ifhmode\unskip\space\fi MR }
\providecommand{\MRhref}[2]{%
  \href{http://www.ams.org/mathscinet-getitem?mr=#1}{#2}
}
\providecommand{\href}[2]{#2}

 \end{document}